\documentclass[12pt]{amsart} 
\usepackage{amssymb}
\usepackage{graphicx}

\theoremstyle{plain}
\newtheorem{theorem}{Theorem}
\newtheorem{lemma}[theorem]{Lemma}
\newtheorem{pro}[theorem]{Proposition}
\newtheorem{cor}[theorem]{Corollary}

\theoremstyle{definition}

\newtheorem{remark}[theorem]{Remark}

\begin{document}
\title{On sofic approximations of
 Property (T) groups}

\author{G\'abor Kun}

\email{kungabor@renyi.hu}

\thanks{This research was supported by the Marie Curie IIF Fellowship Grant 627476, by the ERC Consolidator Research Grant No. 648017, by the National Research, Development and Innovation Office (NKFIH) Grant ERC HU$_{15}$ 118286 and by the MTA Momentum Random Spectra research grant.}

\maketitle

\begin{abstract}
We prove Bowen's conjecture that every sofic approximation of a countable group with Kazhdan Property (T) is essentially a vertex-disjoint union of expander graphs. 
We characterize graph sequences that are essentially a vertex-disjoint union of expander graphs in terms of the Markov operator.
\end{abstract}

\section{Introduction}

A sequence of finite graphs is locally convergent if for every $r$ the isomorphism class of a rooted $r$-ball centered at a vertex chosen uniformly at random converges in distribution.
A group is called \textbf{sofic} if any of its labeled Cayley graphs admits a sofic approximation, this is, a local approximation by finite labeled graphs. Sofic groups were introduced by Gromov \cite{Gro}, see also Weiss \cite{We}. Many classical conjectures not known in general hold for the class of sofic groups: Gottschalk's conjecture (Gromov \cite{Gro}), Kaplansky's direct finiteness conjecture (Elek, Szab\'o \cite{ESdf}) and Connes' embedding conjecture (Elek, Szab\'o \cite{EShyp}). For more on sofic groups see Pestov \cite{P}, Capraro and Lupini \cite{CL}. It is a major open problem if every group is sofic, though it is widely believed that non-sofic groups exist. The main result of this paper, Theorem~\ref{Bowen} and its proof has already turned out a useful tool to construct inapproximable group actions: The author and Thom have used this to prove that a quite wide class of group actions has no local-global approximation (defined by Hatami, Lov\'asz and Szegedy \cite{HLSz}) by finite graphs. Our main result is the following theorem conjectured by Bowen \cite{B}.

\begin{theorem}\label{Bowen}
Let $\Gamma$ be a countably infinite Property (T) group and $\{ G_n\}_{n=1}^{\infty}$ a sofic approximation of $\Gamma$. 
Then there exists a $\gamma>0$ and a sequence of finite $d$-regular graphs $\{ G'_n \}_{n=1}^{\infty}$ such that

\begin{enumerate}

\item
{$V(G_n) = V(G'_n)$}

\item
{$\lim_{n \rightarrow \infty} \frac{|E(G_n) \Delta E(G'_n)|}{|V(G_n)|}=0$}

\item
{There exists $\gamma>0$ that for every $n$ the graph $G'_n$ is a vertex-disjoint union of $d$-regular graphs with Cheeger constant at least $\gamma$.}

\end{enumerate}

\end{theorem}

\begin{remark}
Theorem~\ref{Bowen} gives an ergodic decomposition theorem for certain non-separable probability measure spaces on the ultraproducts of these graphs with an invariant group action as pointed out by Bowen \cite{B}.
\end{remark}



We will give a more detailed description of graph sequences that are essentially disjoint union of expander graphs. Even though this seems to be a complete, ultimate description of a graph sequence a refinement of this picture turns out very useful, especially in case of sofic approximations of Kazhdan groups. We can give a characterization in terms of the Markov operator. $M$ denotes the Markov operator, and $\|*\|$ denotes the $L_2$ norm and $\| * \|_{\infty}$ the $L_{\infty}$ norm with respect to the counting measure on finite sets of vertices in a finite graph. Theorem~\ref{graphequiv} is the main tool in the proof of Theorem~\ref{Bowen}.

\begin{theorem}\label{graphequiv}
Let $\{ G_n \}_{n=1}^{\infty}$ be a sequence of $d$-regular graphs. Then the followings are equivalent:

\begin{enumerate}

\item
{There exists an $\varepsilon>0$ such that for every $\alpha>0$ for all, but finitely many $n$ and for every function
$f: V(G_n) \rightarrow \mathbb{R}$ the inequality $\| M_n^4f - M_n^2f \| \leq (1-\varepsilon) \| M_n^2f-f \| + \alpha|V(G_n)|^{\frac{1}{2}} \|f \|_{\infty}$ holds.}

\item
{There exists a $C>0$ such that for every $\alpha>0$ for all, but finitely many $n$ there exists a  set $B_n \subseteq V(G_n), |B_n| < \alpha |V(G_n)|$ such that
for every $T \subseteq V(G_n) \setminus B_n$ if $|\partial T| < C|T|$ then there exists a set $U$ such that $|U \triangle T| < |T|/2$ and $|\partial U| < \alpha |U|.$ 
}

\item
{There exists a $\gamma>0$ such that for every $\alpha>0$ for all, but finitely many $n$ there exists a partition $V(G_n)=\cup_{i=1}^{a_n} P^i_n$
such that $|P^0_n|<\alpha|V(G_n)|, \sum_{i=1}^{a_n} |\partial P^i_n| < \alpha |V(G_n)|$, and for every $1 \leq i \leq a_n$ and
$S \subseteq P^i_n$, where $|S| \leq |P^i_n|/2$ the inequality $|\partial S| \geq \gamma |S|$ holds.
}

\item
{There exists a $\gamma>0$ and a sequence of $d$-regular graphs $\{ \underline{G}_n \}_{n=1}^{\infty}$ such that $V(G_n) = V( \underline{G}_n)$,
$\lim_{n \rightarrow \infty} \frac{|E(G_n) \triangle E(\underline{G}_n)|}{|V(G_n)|}=0$ and every $\underline{G}_n$ is a vertex-disjoint union of graphs with Cheeger constant at least $\gamma$.}

\end{enumerate}

\end{theorem}

\begin{remark}
We know much more in the case when the sequence is the sofic approximation of a Kazhdan group: $\varepsilon$ can be chosen to be the Kazhdan constant in $(1)$. 
The exceptional small set $B_n$ in $(2)$ and $P_n^0$ in $(3)$ can be contained by the set of vertices which give the error of the sofic approximation, see Proposition~\ref{soficrounding}. 
\end{remark}

We can give a similar description of graph sequences that are essentially the disjoint union of expander graphs far from bipartite graphs, this is, with no large negative eigenvalues. 

\begin{cor}
Let $\{ G_n \}_{n=1}^{\infty}$ be a sequence of $d$-regular graphs. Then the followings are equivalent:

\begin{enumerate}

\item
{There exists an $\varepsilon>0$ such that for every $\alpha>0$ for all, but finitely many $n$ and for every function
$f: V(G_n) \rightarrow \mathbb{R}$ the inequality $\| M^2f-Mf \| \leq (1-\varepsilon) \| Mf-f \| + \alpha|V(G_n)|^{\frac{1}{2}} \|f \|_{\infty}$ holds.}

\item
{There exists a $\lambda>0$ and a sequence of $d$-regular graphs $\{ G'_n \}_{n=1}^{\infty}$ such that $V(G_n) = V(G'_n)$,
$\lim_{n \rightarrow \infty} \frac{|E(G_n) \triangle E(G'_n)|}{|V(G_n)|}=0$ and every $G'_n$ is a vertex-disjoint union of $d$-regular graphs such that all but one eigenvalue of 
every such graph is in the interval $(-\lambda, \lambda).$}

\end{enumerate}

\end{cor}

\begin{remark}
If for $\varepsilon>0$ a finite regular graph $G$ satisfies for every function $f: V(G) \rightarrow \mathbb{R}$ the inequality 
$\| M^2f-Mf \| \leq (1-\varepsilon) \| Mf-f \|$ then $G$  is a vertex-disjoint union of 
$d$-regular graphs such that all but one eigenvalue of every such graph is in the interval $(\varepsilon-d, d-\varepsilon).$
\end{remark}

We give further applications in the last section: We show that a sequence of $2$-dimensional complexes that locally converges to the Cayley complex of a finitely presented 
Property (T) group is not $1$-hyperfinite: This gives an alternative to the construction of Freedman and Hastings \cite{FH}. We answer a question of L. M. Lov\'asz on Lipschitz
embeddings of graphs into large girth graphs and give a new proof of Theorem 9.1. of Mendel and Naor \cite{MN} on coarse embeddings into large girth graphs.\footnote{The theorems of the last section have recently been generalized by the author and Gaboriau to every non-treeable group.}

\section{Definitions}

We deal with sequences $\{G_n\}_{n=1}^{\infty}$ of finite, undirected, $d$-regular graphs.
Finite graphs are equipped with the counting measure, where the measure of every vertex is $1$.
Let $\| * \|, \|*\|_{\infty}$ and $\| * \|_1$ denote the $L_2, L_{\infty}$ and $L_1$ norms, respectively.
$M_G$ will denote the Markov operator of the graph $G$. We will usually drop the subscript.
Given a subset $S \subseteq V(G)$ define its edge boundary as $\partial(S)=E(S,V(G) \setminus S)$.
The Cheeger constant of a finite, undirected graph $G$ is $h(G)=min_{S \subseteq V(G), |S| \leq |V(G)|/2} \frac{|\partial S|}{|S|}$.
We say that the sequence $\{G_n\}_{n=1}^{\infty}$ is {\bf essentially a disjoint union of expander graphs}
if there exists $\gamma>0$ such that every $G_n$ can be turned into a vertex-disjoint union of $d$-regular graphs with Cheeger constant at least $\gamma$ 
on the same set of vertices after removing and adding $o(|V(G_n)|)$ edges.


We say that the finitely generated group $\Gamma$ has Kazhdan Property (T) if there is a finite set of generators
$S$ and an $\kappa>0$ such that for every Hilbert space $\mathcal{H}$ and $\pi: \Gamma \rightarrow U(\mathcal{H})$
unitary representation of $\Gamma$ either $\pi$ has a non-zero, invariant vector, or for $A = \sum_{s \in S} \pi(S)/|S|$
the inequality $\| Ah \| \leq (1-\kappa) \| h \|$ holds for every $h \in \mathcal{H} \setminus \{ \underline{0} \}$. See the
book of Bekka, de La Harpe and Valette on Property (T) \cite{BHV}.
We will use the following consequence of the Kazhdan Property.

\begin{lemma}\label{Hilbert}
Let $\Gamma$ be a finitely generated group with Kazhdan Property $(T)$ with respect to $\kappa>0$ and a finite set of
generators $S \subseteq \Gamma$, where $1 \in S$. Let $\mathcal{H}$ be a Hilbert space, $\pi: \Gamma \rightarrow U(\mathcal{H})$ a
unitary representation of $\Gamma$. Set $A = \sum_{s \in S} \pi(S)/|S|$. Then for every $h \in \mathcal{H}$ the inequality
$\| A^2h-Ah \| \leq (1-\kappa) \| Ah-h \|$ holds.
\end{lemma}
\begin{proof}
Consider the set of fixed points $\mathcal{F}= \{ h \in \mathcal{H}:
\pi(\gamma)h = h \text{ } \forall \gamma \in \Gamma\}$. $\mathcal{F}$ is a closed
subspace of $\mathcal{H}$. The orthogonal complement of $\mathcal{F}$,
$\mathcal{F}^{\perp}$ is a closed subspace invariant under $\pi(\gamma)$ for
every $\gamma$, since $\pi$ is a unitary representation. Hence
$\mathcal{F}^{\perp}$ is invariant under $A$. These yield that
$(Ah-h) \in \mathcal{F}^{\perp}$ for every $h \in \mathcal{H}$.

The restriction of the representation $\pi$ to $\mathcal{F}^{\perp}$ induces a
representation that does not have any fixed point but $0$, hence
$\| A g \| \leq (1-\kappa) \| g \|$ holds for every $g \in \mathcal{F}^{\perp}$.
We conclude that $\| A^2h - Ah\| =\| A (Ah-h) \| \leq (1-\kappa)\| Ah - h\|$.
\end{proof}

\section{Proofs}

\begin{lemma}\label{rounding}
Let $G$ be a finite $d$-regular graph, $M$ denote the Markov operator, $f:V(G) \rightarrow \mathbb{R},$ and $0<a<b<1$.
Then there exists a $t \in (a;b)$ such that the set $U_t=\{x: f(x)>t \}$ satisfies
$|\partial U_t|^2 \leq \frac{8d^2}{a(b-a)^2}  \|f\|_1 \|f\| \|Mf-f\|.$
\end{lemma}

\begin{proof}
Note that $2\|f\| \|Mf-f\| \geq (\|f\|+\|Mf\|)(\|f\|-\|Mf\|) = \|f\|^2 - \|Mf\|^2 = 
\sum_{x \in V(G)} f(x)^2 - \sum_{x \in V(G)} Mf(x)^2 = \\
1/d \sum_{x \in V(G)} \sum_{y: (x,y) \in E(G)} f(y)^2-Mf(x)^2 = \\
1/d \sum_{x \in V(G)} \sum_{y: (x,y) \in E(G)} (f(y)-Mf(x))^2.$

We use again the triangle inequality: \\
$\big( \sum_{x \in V(G)} \sum_{y: (x,y) \in E(G)} (f(x)-f(y))^2 \big)^{\frac{1}{2}} \leq  \\
\big( \sum_{x \in V(G)} \sum_{y: (x,y) \in E(G)} (f(y)-Mf(x))^2 \big)^{\frac{1}{2}} + \\ 
\big( \sum_{x \in V(G)} \sum_{y: (x,y) \in E(G)} (Mf(x)-f(x))^2 \big)^{\frac{1}{2}} \leq \\
(2d)^{\frac{1}{2}} \|f\|^{\frac{1}{2}} \|Mf-f\|^{\frac{1}{2}} + d^{\frac{1}{2}} \|Mf - f\|.$

Let us choose $t \in (a;b)$ uniformly at random. Applications of the triangle and Cauchy-Schwarz inequalities imply \\
$\mathbb{E}_t^2 |\partial U_t| \leq \big(\sum_{x \in V(G), f(x)>a} \sum_{y: (x,y) \in E(G)} \frac{|f(x)-f(y)|}{b-a} \big)^2 \leq \\
d|\{x: f(x)>a \}| \sum_{x \in V(G), f(x)>a} \sum_{y: (x,y) \in E(G)} \big( \frac{f(x)-f(y)}{b-a} \big)^2 < \\
\frac{d\|f\|_1}{a} \frac{d}{(b-a)^2} (2^{\frac{1}{2}}\|f\|^{\frac{1}{2}} \|Mf-f\|^{\frac{1}{2}} +  \|Mf - f\|)^2 \leq
\frac{d\|f\|_1}{a} \frac{d}{(b-a)^2} (4\|f\| \|Mf-f\|^2 +  2\|Mf-f\|^2) \leq \frac{8d^2}{a(b-a)^2}  \|f\|_1 \|f\| \|Mf-f\|.$ 
The lemma follows.
\end{proof}

\begin{cor}~\label{kov}
Let $G$ be a finite $d$-regular graph, $M$ denote the Markov operator and $T \subseteq V(G)$. Assume that there is an integer $K$ and 
$\kappa, \varepsilon>0$ such that $\|M^{k+1}\chi_T - M^k \chi_T \| \leq \big( (1-\kappa)^k + \varepsilon \big) \|M\chi_T-\chi_T\|$ holds for every $1 \leq k \leq K$.
Then there exists a set $U$ such that\\
 $|\partial U|^2 \leq 432  d^2 |T|^2 \big( (1-\kappa)^K + \varepsilon \big)$ and \\
$|U \triangle T| \leq 9\big( \frac{1}{\kappa}+K\varepsilon \big)^2 \|\chi_T - M\chi_T \|^2$.  
\end{cor}

\begin{proof} 
We apply Lemma~\ref{rounding} to $f=M^K \chi_T$ and $(a;b)=(\frac{1}{3}; \frac{2}{3})$ in order to get a set $U=U_t$ such that 
$|\partial U|^2 \leq 216 d^2 |T|^{\frac{3}{2}} \|M^{K+1}\chi_T-M^K\chi_T\| \leq  216 d^2 |T|^{\frac{3}{2}} \big( (1-\kappa)^K + \varepsilon \big) \| M\chi_T - \chi_T \|
\leq 432  d^2 |T|^2 \big( (1-\kappa)^K + \varepsilon \big)$.

We bound $|U \triangle T|$ using the triangle inequality:
$\| M^K\chi_T - \chi_T \| \leq \sum_{k=0}^{K-1} \| M^{k+1}\chi_T - M^k \chi_T \| \leq  \sum_{k=0}^{K-1} \big( (1-\kappa)^k +\varepsilon \big) \|M\chi_T-\chi_T\| \leq (\frac{1}{\kappa} + K\varepsilon ) \| M\chi_T - \chi_T \|$. Note that $M^K\chi_T (x) < \frac{2}{3}$ for every $x \in T \setminus U$, while 
$M^K\chi_T(x) > \frac{1}{3}$ for every $x \in U \setminus T$. Hence  $|U \triangle T| \leq 9 \| M^K\chi_T - \chi_T \|^2$. The lemma follows.
\end{proof}

\begin{lemma}~\label{soficremoval}
Let $\Gamma$ be a countable Kazhdan Property (T) group with a fixed finite set of generators $S$, where $1 \in S$, and let $\kappa$ denote the Kazhdan constant.
For every $\varepsilon>0$ and positive integer $k$ there exists an integer $r$ such that for every finite, regular, edge labeled (with the elements of $S$) graph $G$, 
and subset $T \subseteq V(G)$ if for every $t \in T$ the ball $B(t,r)$ is isomorphic to the $r$-ball in the Cayley graph of $\Gamma$ then the inequality \\ 
$\|M^{k+1}\chi_T - M^k \chi_T \| \leq \big( (1-\kappa)^k + \varepsilon \big) \|M\chi_T-\chi_T\|$ holds. 
\end{lemma}

\begin{proof}
We prove by contradiction. Suppose that there exists an $\varepsilon>0$ such that for every $r$ 
there exists a finite graph $G_r$ and $T_r \subseteq V(G_r)$ such that 
the conditions of the lemma hold, but $\|M^{k+1}\chi_{T_r}-M^k\chi_{T_r}\| > \big((1-\kappa)^k + \varepsilon \big) \|M\chi_{T_r}-\chi_{T_r}\|$.
Consider the measure $\mu_r$ on $V(G_r)$ such that every vertex has measure $\frac{1}{|supp(T_r)|}$, and an ultraproduct $G$ of the graphs
on these measure spaces by a nonprincipial ultrafilter. Then $V(G)$ is also a measure space. Consider the subset $X \subseteq V(G)$ that is the connected component of 
the ultraproduct $T$ of the sets $\{T_r\}_{r=1}^{\infty}$. The labeling of the edges by the elements of $S$ induces a probability-measure preserving, 
free action of $\Gamma$  on the restriction of the measure space $V(G)$ to $X$. This extends to a unitary representation $\pi: \Gamma \rightarrow U(L_2(X))$.
Note that $\|M^{k+1}\chi_T-M^k\chi_T\| \geq \big( (1-\kappa)^k + \varepsilon \big)  \|M\chi_T-\chi_T\| >  (1-\kappa)^k \|M\chi_T-\chi_T\|$. This contradicts Lemma~\ref{Hilbert} and completes the proof of the lemma.
\end{proof}

Lemma~\ref{soficrounding} and Corollary~\ref{kov} imply the following Proposition.

\begin{pro}~\label{soficrounding}
Let $\Gamma$ be a finitely generated Kazhdan group with finite and symmetric generating set $S$, where $1 \in S$, and Kazhdan constant $\kappa$.
For every $\alpha>0$ there exists an integer $r$ such that for any finite, $S$-edge labeled graph $G$ and $T \subseteq V(G)$, where the  ball $B(t,r)$ is isomorphic
to the $r$-ball in the Cayley graph of $\Gamma$ for every $t$ in $T$ there exists a set $U$ such that $$|\partial U | \leq \alpha | \chi_U | \quad \mbox{and} \quad | U \triangle T | \leq \frac{10}{\kappa^2} \|\chi_T - M\chi_T \|^2 \leq \frac{5 |\partial T|}{d\kappa^2}.$$ 
\end{pro}

This shows that the sofic approximation of a Kazhdan group satisfies condition $(2)$ in Theorem~\ref{graphequiv}. Hence Theorem~\ref{graphequiv} implies Theorem~\ref{Bowen}.

\begin{pro}\label{removal}
Let $G$ be a finite $d$-regular graph, $k>0$ an integer, $0<c<c'<1$ and $\alpha,
\delta>0$. Assume that for every subset $S \subseteq V(G)$ the inequality
$\| M^{k+1}\chi_S - M^k\chi_S \| < c \|M \chi_S - \chi_S \| + \delta |V(G)|^{\frac{1}{2}}$ holds.
Then there exists a set of
vertices $B \subseteq V(G)$ of size at most
$\frac{\delta^2(d^{2k+2}+1)}{\alpha^2 (c'-c)^2} |V(G)|$
such that for every $S \subseteq V(G) \setminus B$ \\
either $\| M\chi_S-\chi_S \| < \alpha \| \chi_S \|$, \\
or $\|M^{k+1}\chi_S - M^k\chi_S \| < c' \| M \chi_S - \chi_S \|$ holds.
\end{pro}

\begin{proof}
Let $B'$ be a maximal subset of $V(G)$ under containment that satisfies the inequalities
$\|M^{k+1}\chi_{B'} - M^k\chi_{B'} \| \geq c' \|M \chi_{B'} - \chi_{B'}\|$ and
$\| M\chi_{B'} -\chi_{B'} \| \geq \alpha \| \chi_{B'} \|$. Set $B = N_{2k+2} (B')$ (the set of vertices at distance at most $2k+2$).

We prove the upper bound on the size of $B$. For any subset of vertices $S$ the inequality $|N_{2k+2}(S)| \leq (d^{2k+2}+1) |S|$ holds,
since $G$ is $d$-regular, in particular$|B| \leq (d^{2k+2}+1) |B'|$.
We know that \\
$$(1) \text{     }\| M^{k+1}\chi_{B'} - M^k\chi_{B'} \| < c \|M \chi_{B'} -\chi_{B'} \| + \delta  |V(G)|^{\frac{1}{2}},$$
since this holds for any subset. On the other hand, $B'$ satisfies 
$$(2) \text{     }\| M^{k+1}\chi_{B'} - M^k\chi_{B'} \| \geq c' \|M \chi_{B'} - \chi_{B'} \| \text{ and}$$ 
$$(3) \text{     }\| M \chi_{B'}-\chi_{B'} \| \geq \alpha \| \chi_{B'} \|.$$

\noindent
$(1)$ and $(2)$ imply
$$(c'-c) \|M \chi_{B'} -\chi_{B'} \| < \delta |V(G)|^{\frac{1}{2}}.$$

\noindent
Using $(3)$ we can conclude that 

$|B| \leq (d^{2k+2}+1)|B'|= (d^{2k+2}+1)\|\chi_{B'}\|^2 \leq (d^{2k+2}+1) \frac{\| M \chi_{B'}-\chi_{B'}\|^2}{\alpha^2} \leq \frac{\delta^2(d^{2k+2}+1)}{\alpha^2 (c'-c)^2} |V(G)|$.

Let $S \subseteq V(G) \setminus B$ and consider the set $S \cup B'$.
Note that \\ $\|M^{k+1}\chi_{S \cup B'} - M^k\chi_{S \cup B'} \|^2 = \|M^{k+1}\chi_S -M^k \chi_S \|^2 +
\|M^{k+1}\chi_{B'} - M^k\chi_{B'} \|^2$, since the supports of $M^{k+1}\chi_S - M^k\chi_S$ and $M^{k+1}\chi_{B'} - M^k\chi_{B'}$ are disjoint. Similarly,
$\| M \chi_{S \cup B'} - \chi_{S \cup B'} \|^2 = \| M \chi_S  - \chi_S \|^2 +\| M \chi_{B'} - \chi_{B'} \|^2$.
The maximality of $B'$ implies that either $\| M\chi_S-\chi_S \| < \alpha \| S \|,$ or $\|M^{k+1}\chi_S - M^k\chi_S \| < c' \|M \chi_S - \chi_S \|$.

\end{proof}

We start the proof of Theorem~\ref{graphequiv} with
$(1) \rightarrow (2)$. We consider the $d^2$-regular graph $G_n^2$ with vertex set $V(G_n)$, where two vertices are adjacent if they are connected by a path of length two.
The Markov operator on $G_n^2$ is $M_n^2$. We use the following remark to relate boundary sizes in $G_n$ and $G_n^2$.

\begin{remark}
Consider a $d$-regular graph $G$ and a subset $S \subseteq V(G)$. The $|\partial_G S| \geq  \frac{|\partial_{G^2} S|}{2d}$. On the other hand,
if every $s \in S$ has at least one neighbor in $S$ then $|\partial_G S| \leq  |\partial_{G^2} S|$.
\end{remark}

First we show that (a strengthening of) $(2)$ holds for $G_n^2$. We show that there exists a $C>0$ such that for every $\alpha>0$ for all, but finitely many $n$ there exists a  set $B_n \subseteq V(G_n), |B_n| < \alpha |V(G_n)|$ such that for every $T \subseteq V(G_n) \setminus B_n$ if $|\partial_{G_n^2} T| < C|T|$ then there exists a set $U$ such that $|U \triangle T| < |T|/3$ and $|\partial_{G_n^2} U| < \alpha |U|$. (Note that the strengthening is in the upper bound $|U \triangle T| < |T|/3$.) This is a consequence of Proposition~\ref{removal} and Corollary~\ref{kov} applied to $G_n^2$.

Finally, we need to show that $(2)$ holds for $G_n$, too. We can use the same exceptional set $B_n$ as for $G_n^2$ (for different $\alpha$ and $C$). 
Given a set $T$ with small boundaryin $G_n$ and hence in $G_n^2$ we get a set $U'$ with small boundary in $G_n^2$. But $|\partial_{G_n} U'|$ might be large.
However, if every vertex of $U'$ has a neighbor in $U'$ then $|\partial_{G_n} U'|$ is small as remarked. If a vertex of $U'$ has no neighbors in $U'$ we remove it from $U'$ recursively. 
The resulting set $U$ will satisfy the conditions by the Remark: We could not remove too many vertices from $U$ since ${|\partial_{G_n} U|}$ 
was small and it decreased by $d$ after every removal.

Now we show $(2) \rightarrow (3)$. 
The following process will give the required partition of $V(G_n)$.
We use $(2)$ for $C=\gamma$ and $\alpha'=\alpha/(d+3)$ in order to get a set $B_n$ (for all but finitely many $n$).  Set $P^0_n=B_n$, and for $i \geq 0$ proceed as follows.
If $|\partial T| \geq \gamma|T|$ for every $T \subseteq V(G_n) \setminus \big( \cup_{j=0}^i P^j_n  \big)$ then set $P^{ i+1}_n = V(G_n) \setminus \big( \cup_{j=0}^i P^j_n  \big)$.
If there exists a set $T^i_n \subseteq V(G_n) \setminus \big( \cup_{j=0}^i P^j_n  \big)$ such that $|\partial T^i| < \gamma|T^i_n|$ then consider the set $U^{i+1}_n$ given by $(2)$ and
define $P^{i+1}_n=U^{i+i}_n \setminus  \big( \cup_{j=0}^i P^j_n  \big)$. Note that $|P^{i+1}_n| < 2 |T^{i+1}_n|$. 
 
The total expansion of the sets in the partition is \\
$\sum_{i=0}^{a_n} | \partial P^i_n | \leq d|P^0_n| + \sum_{i=1}^{a_n-1} |\partial U^i_n| \leq d \alpha' |V(G_n)| + \alpha' \sum_{1 \leq  i \leq  a_n-1} 3|P^i_n| \leq 
(d+3)\alpha'|V(G_n)|=\alpha |V(G_n)|$.  

On the other hand, $|P^i_n| < 2 |T^i_n|$, hence, by the choice $T^i_n$, the inequality $|\partial S| \geq \gamma |S|$ holds for every $S \subseteq P^i_n$, where $|S| \leq |P^i_n|/2$.

$(3) \rightarrow (4)$ will follow from the next lemma. 

\begin{lemma}
For every $d$ and $\gamma>0$ there exists $\alpha>0$ such that the following holds: 
Consider a $d$-regular graph $G$ and a subset $S \subseteq V(G)$ such that $d|S|$ is even, $E(T,V(G)\setminus T)>\gamma|T|$ if $T \subseteq S, |T| \leq |S|/2$ and 
$|E(S,V(G) \setminus S)| \leq \alpha |S|$. Then there exists a $d$-regular graph $Q$ such that 
$V(Q)=S, |E(Q) \setminus E(G)|=|E(S,V(G) \setminus S)|$ and $|E(T, V(G) \setminus T)| \geq \frac{\gamma}{6} |T|$ 
if $T \subseteq S, |T| \leq |S|/2$. Moreover, if $G$ is regularly edge labelled then $Q$ is regularly edge labelled, too.
\end{lemma}

\begin{proof}
Let $B$ denote the set of vertices in $S$ adjacent to a vertex not in $S$. Set $r=4/\gamma$.
Find a matching $\mathcal{M}$ of size $|E(S,V(G) \setminus S)|/2$ in $S$ such that the distance of any two edges in $\mathcal{M}$ is greater than $2r$, and
no vertex of $B$ is adjacent to an endpoint of an edge in $\mathcal{M}$. Such a matching $\mathcal{M}$ exists if $\alpha$ is small enough. 
Let $M$  denote the set of vertices covered by $\mathcal{M}$.
Consider a bijection $b: E(S,V(G)\setminus S) \rightarrow M$. Let $Q$ denote the following graph:
$V(Q)=S, E(Q)= \{ (x,b((x,y))): x \in S, y \in V(G) \setminus S, (x,y) \in E(G)\} \cup \{(x,y) \in E(G): x,y \in S \} \setminus \mathcal{M}$.

We prove that $Q$ is a $\frac{\gamma}{6}$-expander.
Consider a subset of vertices $T \subseteq S$ of size at most $\frac{|S|}{2}$.
Note that $|E_Q(T ,S \setminus T) \cap E_G(T, S \setminus T)| \geq |E_G(T, V(G_n) \setminus T)| - |T \cap M| - \sum_{t \in T \cap B} |N(t)\setminus S|$.

On the other hand $|E_Q(T, S \setminus T) \setminus E_G(T ,S \setminus T)| =
|\{ (x,y): x \in B \cap T, y \in M \setminus T, \exists z \in V(G) \setminus S \text{ } (x,z) \in E(G), \text{ } b((x,z)) = y\}| 
\geq \sum_{t \in T \cap B} |N(t)\setminus S| - |T \cap  M|$. Together these imply $|E_Q(T ,S \setminus T)| \geq |E_G(T, V(G_n) \setminus T)| - 2|T \cap M|$.

Consider the intersection $S \cap B(x,r)$ for every $x \in T \cap M$:
If it has less than $r$ vertices then it contains the endpoint of at least one edge in $|E_Q(T, S \setminus T)|$.
Hence $|E_Q(T ,S \setminus T)| \geq |T \cap M| - |T|/r$. Altogether, $3|E_Q(T ,S \setminus T)| \geq  |E_G(T, V(G_n) \setminus T)| -2|T|/r \geq (\gamma-2/r) |T| = \frac{\gamma}{2} |T|$.

\end{proof}

Finally, we prove $(4) \rightarrow (1)$. For the Markov operator $\underline{M}_n$ of $\underline{G}_n$ the inequality $\| \underline{M}^4_n f-\underline{M}^2_n f \| \leq (1-\varepsilon) \| \underline{M}_n^2f-f \|$ holds, where $\varepsilon= \frac{\gamma^2}{2d^2}$ follows from the bound on the eigenvalue gap proved by Dodziuk \cite{D} and independently by Alon and Milman, see \cite{AS}. The additional term $\alpha|V(G_n)|^{\frac{1}{2}} \|f \|_{\infty}$ bounds the norm of the differences if we exchange $\underline{M}_n$ by $M_n$ (and $n$ is large enough):  \\
$\|\underline{M}_n^2f-f \| - \|M_n^2f-f \| \leq \|(\underline{M}_n^2-M_n^2)f \| \leq \| (\underline{M}_n - M_n)\underline{M}_n f \| + \| M_n(\underline{M}_n - M_n)f \| \leq 
|E(G_n) \triangle E(\underline{M}_n)|^{\frac{1}{2}} \| \underline{M}_n f \|_{\infty} +  \| (\underline{M}_n - M_n)f \| \leq 
3 |E(G_n) \triangle E(\underline{M}_n)|^{\frac{1}{2}} \|f\|_{\infty} = o(|V(G_n)|^{\frac{1}{2}} \|f\|_{\infty})$. The bound
$\|(M_n^4-M_n^2)f\| \leq \|(\underline{M}_n^4-\underline{M}_n^2)f\| + o(|V(G_n)|^{\frac{1}{2}} \|f\|_{\infty})$ can be obtained similarly.

\section{Embeddings}

We apply Theorem~\ref{Bowen} to study \textbf{embeddings} of graph sequences converging to the Cayley graph and sequences of $2$-dimensional CW complexes converging to the
Cayley complex of a Property (T) group. An $r$-ball in a CW complex centered at a $0$-cell $x$ will be the subcomplex spanned by the $0$-cells at distance at most $r$ from $x$,
where the distance is the graph distance in the $1$-skeleton. A sequence of CW complexes is locally convergent if for every $r$ the isomorphism class of a rooted $r$-ball centered at a $0$-cell chosen uniformly at random converges in distribution.

\begin{theorem}\label{gen}
Consider a sequence of finite CW complexes $\{ X_n \}_{n=1}^{\infty}$ that converges to a Cayley complex $\mathcal{C}$ of a countably infinite Property (T) group.
And consider a sequence of finite $1$-dimensional CW complexes $\{ Y_n \}_{n=1}^{\infty}$ and continuous mappings $\{ f_n: X_n \rightarrow Y_n \}_{n=1}^{\infty}$.
Then for every $K$ if $n$ is large enough then there exists a $1$-cell in $X_n$ whose image intersects the image of at least $K$ other $1$-cells.
\end{theorem}

This is an alternative construction of $2$-dimensional \textbf{non-hyperfinite simplicial complexes} in the sense of Freedman and Hastings \cite{FH}.
Informally speaking, a $2$-dimensional simplicial complex is $1$-hyperfinite if for every $\varepsilon>0$ it admits after the removal of an $\varepsilon$-proportion of
the $0$-cells and the higher dimensional cells containing these a continuous mapping to a $1$-dimensional complex such that the pre-image of every point has bounded
diameter (depending on $\varepsilon)$.

\begin{cor}
Consider a sequence of finite CW complexes that converges to a Cayley complex of a finitely presented infinite Property (T) group, and a sequence
$\{ X_n \}_{n=1}^{\infty}$ of simplicial complexes obtained by the subdivision of the $2$-faces of the CW complexes into a bounded number of simplices.
Then $\{ X_n \}_{n=1}^{\infty}$ is not $1$-hyperfinite.
\end{cor}

Theorem~\ref{gen} has a surprising graph theoretical corollary, too. This answers a question of L. M. Lov\'asz.

\begin{cor}\footnote{An ineffective version of the corollary can be obtained using that Property (T) groups admit no treeable almost free action \cite{Tvstree}
and that every subrelation (refinement) of a treeable relation is treeable \cite{G}. This holds also in the case when $\Gamma$ is not finitely presented.}
Consider a sequence of finite graphs $\{ G_n \}_{n=1}^{\infty}$ that converges to a Cayley graph corresponding to a finite presentation of a Property (T) group $\Gamma$.
The sequence $\{ G_n \}_{n=1}^{\infty}$ admits no $L$-Lipshitz embedding for any $L$ into a sequence of graphs with bounded degree and girth greater than $rL$,
where $r$ is the length of the shortest relation in the presentation of $\Gamma$.
\end{cor}

\begin{proof}
Consider a sequence of large girth graphs, an integer $L$ and a sequence of $L$-Lipschitz mappings. Let $\{ X_n \}_{n=1}^{\infty}$ denote the following sequence
of $2$-dimensional CW-complexes: The $1$-skeleton of $X_n$ is $G_n$ and for every cycle of length at most $r$ there is a $2$-cell on this cycle. Let $\{ Y_n \}_{n=1}^{\infty}$
denote the sequence of $1$-dimensional CW-complexes. And let $\{ f_n: X_n \rightarrow Y_n \}_{n=1}^{\infty}$ a sequence of continuous mappings extending the sequence of the
$L$-Lipschitz mappings: Every $1$-cell of $X_n$ will be mapped to the shortest path connecting the image of its two $0$-subcells. The image of every short cycle will be
null-homotopic by the girth condition, hence the mapping can be extended to $2$-cells. Theorem~\ref{gen} can be applied: The theorem follows, since the degrees are bounded.
\end{proof}

\subsection{Proofs}

\begin{lemma}\label{pfold}
Consider a sequence of finite $d$-regular graphs $\{ G_n \}_{n=1}^{\infty}$. Assume that $\{ G_n \}_{n=1}^{\infty}$ locally converges to the Cayley
graph of a finitely generated infinite group $\Gamma$ with Property (T). Consider a positive integer $L$, a prime $p$ and a mapping
$\varphi_n: E(G_n) \rightarrow \{ -L, \dots , L \} \subseteq \mathbb{Z}_p$ for every $n$. Assume the followings.

\begin{enumerate}

\item{$\varphi((x,y))=-\varphi((y,x))$ for every edge $(x,y)$.}

\item{Given an integer $l$ consider the sum of $\varphi_n$ over every cycle of length $l$.
Assume that for every $l$ the proportion of cycles with nonzero sum goes to zero as $n$ goes to infinity.}

\item{Given an integer $t>0$ let $s^n_0, \dots , s^n_t$ be a random walk on $G_n$ chosen uniformly at random. Assume that the limit distribution
$lim_{t \rightarrow \infty} lim_{n \rightarrow \infty} \sum_{i=1}^t \varphi_n((s^n_{i-1},s^n_i))$ exists, where we consider the statistical distance,
and it is equal to the uniform distribution on $\mathbb{Z}_p$.}

\end{enumerate}

Then $p \leq 1 + \frac{4dL}{\gamma}$, where $\gamma>0$ is given in Theorem~\ref{Bowen}.
\end{lemma}

\begin{proof}
We may assume that the sequence $\{ G_n \}_{n=1}^{\infty}$ is expander. The mapping $\varphi_n$ induces a $p$-fold covering of $G_n$, denote the covering graph by $C_n$.
The vertices of the graph $C_n$ are pairs $(x,y)$, where $x \in G_n$ and $y \in \mathbb{Z}_p$, and a pair of vertices $(x,y)$ and $(x',y')$ is an edge if
$x$ and $x'$ are adjacent and $y-y' = \varphi_n((x, x'))$. Now $\mathbb{Z}_p$ embeds into the automorphism group of $C_n$ acting on the second coordinate.
The sequence $\{ C_n \}_{n=1}^{\infty}$ converges still to the same Cayley graph, this is the point where we use $(2)$. Hence by Theorem~\ref{Bowen} it is essentially
a disjoint union of expanders. It is easy to see that it is either essentially an expander or essentially a disjoint union of $p$ expanders isomorphic to $G_n$,
where the $\mathbb{Z}_p$-action is essentially permuting these subgraphs.

First suppose that (an infinite subsequence of) $\{ C_n \}_{n=1}^{\infty}$ is essentially an expander sequence: There is a $\gamma>0$
given by Theorem~\ref{Bowen} such that for every subset $S \subseteq C_n$, where $|S|<\frac{1}{2}|C_n|$ we have $|E(S, S^c)| \geq \gamma |S| + o(|C_n|)$.
Consider the set $S=V(G_n) \times \{ 1, \dots , \frac{p-1}{2}\} \subset V(C_n)$.
Clearly $|S|=\frac{p-1}{2p} |V(C_n)|$. Since $|E(S, S^c)| \leq 2dL|V(G_n)|$, we can conclude that $p \leq \frac{4dL}{\gamma} + 1$.

Now suppose that (an infinite subsequence of) $\{ C_n \}_{n=1}^{\infty}$ is essentially a disjoint union of $p$ expanders isomorphic to $G_n$.
Consider one of these expander graphs $D_n$ essentially isomorphic to $G_n$. Note that $(1)-(3)$ still holds for $D_n$. For every $z \in \mathbb{Z}_p$ we
have $|(V(G_n) \times \{ z \}) \cap V(D_n)| = (\frac{1}{p} + o(1))|D_n|$ by $(2)$.

Consider the set $S=(V(G_n) \times \{ 1, \dots , \frac{p-1}{2}\}) \cap D_n$. Clearly $|S|=(\frac{p-1}{2p} + o(1)) |V(D_n)|$.
Since $|E(S, S^c)| \geq \gamma |S| + o(|V(D_n)|)$ and $|E(S, S^c)| \leq \frac{2dL}{p} |V(D_n)| + o(1)$ we conclude that $p \leq 1 + \frac{4dL}{\gamma}$.
\end{proof}

\begin{proof}(of Theorem~\ref{gen})
We prove by contradiction. Consider a sequence of finite CW complexes $\{ X_n \}_{n=1}^{\infty}$ that converges to a Cayley complex $\mathcal{C}$ of a
finitely generated infinite Property (T) group, a sequence of finite $1$-dimensional CW complexes $\{ Y_n \}_{n=1}^{\infty}$, continuous mappings
$\{ f_n: X_n \rightarrow Y_n \}_{n=1}^{\infty}$ and an integer $K$. Suppose for a contradiction that the image of every $1$-cell in $X_n$ intersects the image
of at most $K$ other $1$-cells. We think of $Y_n$ as a graph and use the graph theoretical terminology. We may assume that the image of every $0$-cell is a vertex
and the image of every $1$-cell is a path, since we can change $f$ by a deformation to achieve these.

We choose a prime $p$ later. We assign an element of $\{ -1, +1 \} \subset \mathbb{Z}_p$ independently at random (with probability $\frac{1}{2}$ each) to a subset
of edges of $Y_n$: We will call such edges weighted. To every nonempty set of $1$-cells in $X_n$ whose image intersects in an edge at least we will assign such a
weighted edge in $Y_n$, and to every such maximal (by containment) set of $1$-cells we will assign a different edge. The image of every $1$-cell contains less than $2^K$ weighted edges.

Let $G_n$ denote the $1$-skeleton of $X_n$, $G_n$ is a graph. We show that the image of a long uniform random walk in $G_n$ under $f_n$ contains many weighted edges with
high probability (using that the image of every $1$-cell in $X_n$ intersects the image of at most $K$ other $1$-cells).

\textbf{Claim:} For every $k$ there is a $t$ such that the image of a uniform random walk of length at least $t$ in $G_n$ (with a uniformly random starting vertex)
contains at least $k$ weighted edges with probability at least $(1-1/k)$ if $n$ is large enough (depending on $k, K$).

\begin{proof}
Choose a starting vertex $s_0$. Consider the image of the walk under $f_n$. And consider the path between $f_n(s_0)$ and $f_n(s_t)$ obtained by the removal of the cycles of
this walk in $Y_n$. If this path has at most $k$ weighted edges then it can be covered by the image of at most $k$ $1$-cells. If there is no point whose pre-image intersects
at least $K$ $1$-cells then there are less than $K^{k+1}$ possible images of $0$-cells reachable via such a short sequence from $f_n(s_0)$, and the probability that the
endpoint is the pre-image of any of these $0$-cells can be arbitrary small if $t$ and $n$ are large enough.
\end{proof}

Consider the mapping $\varphi_n: E(G_n) \rightarrow \mathbb{Z}_p$ that assigns to every edge of $G_n$ ($1$-cell of $X_n$) the sum of these weights over the edges in its image
(with orientation and multiplicity).

The values assigned to the weighted edges are chosen uniformly at random from $\{ -1, +1 \} \subset \mathbb{Z}_p$ independently. The sum over every path will be distributed
identically to the endpoint of a random walk on the cycle of length $p$, where the number of steps equals to the number of weighted edges on the path. This will converge to
the uniform distribution as the length goes to infinity. And the value for paths with disjoint images will be independent, hence we have a concentration. Note that so far we
have only used that $X_n$ is large enough and connected.

Choose a prime $p>4\frac{d2^K}{\gamma}+1$. For every length $l$ the proportion of cycles with nonzero sum can be arbitrary small if $n$ is large enough, since $\mathcal{C}$ is simply
connected and the sum over null-homotopic cycles is zero. Hence condition $(2)$ of Lemma~\ref{pfold} is satisfied. The distribution of the sum over uniform random walks in
condition $(3)$ can be arbitrarily close to uniform on $\mathbb{Z}_p$ if $n$ is large enough, since $X_n$ is large enough and connected. Lemma~\ref{pfold} gives a contradiction,
the theorem follows.
\end{proof}

\end{document}